\documentclass[12pt]{article}
\usepackage{ayv_paperstyle}
\newcommand{\prd}{\mathcal{P}^2(\rd)}
\newcommand{\pK}{\mathcal{P}^2(\mathcal{K})}

\begin{document}
	\title{Approximation of deterministic mean field type control systems}
	\author{Yurii Averboukh}
	\date{}
	
\AtEndDocument{
	\vspace{10pt}
	\begin{tabular}{ll}
		Yurii Averboukh:  &
		Krasovskii Institute of Mathematics and Mechanics,  \\ & 
		16, S. Kovalevskaya str,  Yekaterinburg, Russia \\ & 
		\email{ayv@imm.uran.ru}
	\end{tabular}}
	\maketitle
\begin{abstract}
	The paper is concerned with the approximation of the deterministic the mean field type control system by a mean field  Markov chain. It turns out that the  dynamics of the distribution in the approximating system is described by a system of ordinary differential equations. Given a strategy for the Markov chain, we explicitly construct a control in the deterministic mean field type control system. Our method is a realization of the model predictive approach.  The converse construction is also presented. These results lead to an estimate of the Hausdorff distance between the bundles of motions in the deterministic mean field type control system and the mean field Markov chain. Especially, we pay the attention to the case when one can approximate the bundle of motions in the mean field type system by  solutions of a finite systems of ODEs.
	\msccode{93A14,93B11,93C10,93E10}
\end{abstract}
\section{Introduction}

The paper studies mean field type control systems. These systems describe an evolution of many identical agents who play cooperatively and interact via some external field. The mean field type control systems appear within the modeling of swarm of robots, pedestrian flows, etc \cite{Colombo2005,Colombo2011,Cristiani2014,Bellomo2012,Bullo2009}. 

The concept of mean field models comes back to the model of plasma proposed by Vlasov in 1938~\cite{Vlasov,Vlasov_book} and was formalized within the theory of McKean-Vlasov equation~\cite{McKean,Sznitman}. We will focus on the deterministic (first-order) mean field type systems. This means that the dynamics of each agent is described by an ordinary differential equation on the Euclidean space with the right-hand side depending on his/her state, control and the distribution of all agents. 

Notice that the settings of the control theory include the study of optimal control problems as well as the examination of the qualitative properties of the bundle of trajectories. 

The mean field type optimal control theory started with paper~\cite{ahmed_ding_controlld}. Nowadays, for the second order mean field type control system, i.e., when the dynamics of each agent is determined by a stochastic differential equation, the necessary and optimality conditions  are derived (see~\cite{Bayraktar_Cosso_Pham_randomized,Buckdahn_Boualem_Li_PMP_SDE,Carmona_Delarue_PMP,Djehiche_Hamdine,Lauriere_Pironneau_DPP_MF_control,Pham_Wei_2015_Bellman,Pham_Wei_2015_DPP_2016}). The case of first-order mean field type optimal control problems was studied in papers~\cite{Jimenez_Marigonda_Quincampoix,Pogodaev,averboukh_khlopin}, where the variants of dynamic programming and Pontryagin maximum principle were obtained.

The qualitative theory for mean field type control systems studies the general properties of  bundle of motions as well as viability theory issues. For the deterministic mean field type control systems, the existence theorem was proved under the general assumption on the dynamics~\cite{Bivas2021a,Averboukh2022,Jimenez_Marigonda_Quincampoix}, while the Filippov and relaxation theorem are derived under additional assumption of continuity of the vector  field~\cite{Bonnet2021,bonnet2023}. The viability theorem was obtained in the terms of tangent cones (see~\cite{Averboukh2018a,Bonnet2022}) and in the terms of proximal normal cones~\cite{Averboukh2021}.

The main object of the paper is an approximation of the bundle of motions of the first order mean field type control system. Our approach is based on so called Markov approximations. They appear when one replace the ODE determining the dynamics of each agent by a continuous-time mean field Markov chain. The latter can be regarded as a system of infinitely many similar agents with the dynamics determined by  transition rates depending, in particular, on a current distribution of agents. In this case, the dynamics of the whole distribution of agents is described by a nonlinear Kolmogorov equation. If, additionally, one can assure that the agents in the original system do not leave a compact space, the phase space for the approximating Markov chain can be chosen to be finite. This leads to an approximation of the first-order mean field type control system by a finite system of ODE. 

The approximation of the deterministic control system by a Markov chain was proposed in~\cite{Averboukh2016}. In that paper, the approximation of the value function of the zero-sum differential game was studied based on stochastic control with guide approach coming back to research of Krasovskii and Kotelnikova~\cite{a4,a5,a6,kras_delay}.  This technique was extended to mean field type differential games in~\cite{Averboukh2021a}. There, based on a modification of the extremal shift rule for the Wasserstein space and control of the guide strategies, the  value function of the mean field type differential game was approximated by a solution of a finite dimensional differential game. 

In the paper, we focus on the approximation of the mean field type control system and implement the model predictive control approach (see \cite{Model_predictive} and reference therein). Notice that for the original deterministic mean field type control system we assume the open-loop strategies those are distribution of pairs consisting of an initial state and a control whereas the approximating mean field Markov chain implies the feedback strategies. The latters can be regarded as a sequence of open loop strategies those work at the appropriate state. 

As it was mentioned above, we use the methodology of the model predictive approach. For the considered deterministic mean field  type control system, this means that, given a feedback strategy in Markov chain, an agent uses   on a small time interval a control borrowed from a state in the approximating Markov chain. The weights of the controls are determined by an optimal plan between the distribution in the original and approximating system. To solve the converse problem which implies the design of a feedback strategy in the Markov chain based on a given distribution of controls in the mean field type system, one can on each small time interval integrate the controls according to the optimal plan between distributions. 

The main results of the paper includes also an  approximation rates of the aforementioned constructions. They depend only  on the original system, distance between the original and approximating systems, fineness of partition and the maximal transition rate in the mean field Markov chain. Notice also that, if the fineness of the partition tends to infinity, the limiting approximation rates are determined only by the original system and the distance between the original and approximating systems. This, in particular, provides the distance between the bundles of motions in the deterministic mean field type control system and a finite dimensional control system. The earlier results gave only one-side estimate. 

The rest of the paper is organized as follows. In Section~\ref{sect:notation}, we introduce the general notation. The assumption on the deterministic mean field type control system as well as the class of admissible controls are presented in Section~\ref{sect:1_order}. The approximating mean field Markov chain is introduced in Section~\ref{sect:Markov_chain}. Here, we define the general system and show the way to construct an approximating Markov chain for the original deterministic  mean field type system. Section~\ref{sect:model_for_1_order} deals with the model predictive control for the  deterministic system.  In this case, we assume that a feedback strategy for the mean field Markov chain is given and derive the approximation rate of the constructed motion in the deterministic mean field type system. The model predictive control for the mean field Markov chain  is discussed in Section~\ref{sect:model_for_Markov_chain}. Finally, we estimate the Hausdorff distance between the bundle of motions in the deterministic mean field type control system and the Markov chain (see Section~\ref{sect:Hausdorff}).

\section{General notation}\label{sect:notation}
\begin{itemize}
	\item If $n$ is a natural number, $X_1,\ldots,X_n$ are sets, $i_1,\ldots,i_k$ are  indices from $\{1,\ldots,n\}$, then $\operatorname{p}^{i_1,\ldots,i_k}$ is a projection from $X_1\times\ldots X_n$ to $X_{i_1}\times\ldots X_{i_k}$, i.e,
	\[p^{i_1,\ldots,i_k}(x_1,\ldots,x_n)\triangleq (x_{i_1},\ldots,x_{i_k}).\]
	\item If $(X,\rho_X)$, $(Y,\rho_Y)$ are Polish spaces, then $C_b(X,Y)$ denotes the space of bounded continuous functions from $X$ to $Y$. If $Y$ is a normed vector space, then $C_b(X,Y)$ is also a normed vector space with usual $\sup$-norm. Moreover, $C_b(X)\triangleq C_b(X,\mathbb{R})$.  
	\item We always endow a Polish space $(X,\rho_X)$ with the Borel $\sigma$-algebra denoted by $\mathcal{B}(X)$. Moreover, $\mathcal{M}(X)$ stands for the set of Borel nonnegative measures, whereas $\mathcal{P}(X)$ denotes the set of all Borel probabilities:
	\[\mathcal{P}(X)\triangleq \{m\in \mathcal{M}(X):m(X)=1\}.\] We consider on  $\mathcal{M}(X)$ the topology of narrow convergence, i.e., a sequence of measures $\{m_n\}_{n=1}^\infty$ converges to a measure $m$ if, for each $C_b(X)$,
	\[\int_X\phi(x)m_n(dx)\rightarrow \int_X\phi(x)m(dx)\text{ as }n\rightarrow\infty.\] Notice that $\mathcal{P}(X)$ is closed w.r.t. the narrow convergence.
	\item If $(\Omega,\mathcal{F})$, $(\Omega',\mathcal{F}')$ are measurable spaces, $m$ is a measure on $\mathcal{F}$, whilst $h:\Omega\rightarrow\Omega'$ is $\mathcal{F}/\mathcal{F}'$-measurable, then we denote by $h\sharp m$ the push-forward measure on $\mathcal{F}'$ defined by the rule: for each $\Upsilon\in\mathcal{F}'$,
	\[(h\sharp m)(\Upsilon)\triangleq m(h^{-1}(\Upsilon)).\]
	\item If $(X,\rho_X)$ and $(Y,\rho_Y)$ are two Polish space, $m$ is a measure on $X$, then we denote by $\Lambda(X,m;Y)$ the set of measures on $X\times Y$ with marginal distribution on $X$ equal to $m$, i.e.,
	\[\Lambda(X,m;Y)\triangleq \big\{\alpha\in\mathcal{M}(X\times Y):\operatorname{p}^1\sharp \alpha=m\big\}.\]    
	\item A function $\beta:X\rightarrow \mathcal{P}(X)$ is called weakly measurable if, for each $\phi\in C_b(Y)$, the function
	\[x\mapsto\int_Y\phi(y)\beta(x,dy)\] is measurable. Furthermore, each weakly measurable function $\beta$ generates a measure $m\star\beta\in \Lambda(X,m;Y)$ by the rule: for each $\phi\in C_b(X\times Y)$,
	\[\int_{X\times X}\phi(x,y)(m\star\beta)(d(x,y))\triangleq \int_X\int_Y\phi(x,y)\beta(x,dy)m(dx).\]
	\item If $\alpha\in \Lambda(X,m;Y)$, then there exists a weakly measurable function  $\beta$ such that $\alpha=m\star\beta$ \cite[III-70]{Meyer}. These function is called a disintegration of the measure $\alpha$ and denoted by $x\mapsto \alpha(\cdot|x)$. The disintegration is unique a.e.   
	\item If $(X,\rho_X)$ is a Polish space, $p\geq 1$, then  $\mathcal{P}^p(X)$ is the space of probabilities on $X$ such that, for some (equivalently, every) $x_*\in X$,
	\[\int_X\rho_X^p(x,x_*)m(dx)<\infty.\] If $X$ is a normed vector space, then we will always choose $x_*=0$ and put
	\[\varsigma_p(m)\triangleq\left[\int_X\|x\|^pm(dx)\right]^{1/p}.\]
	\item The space $\mathcal{P}^p(X)$ is endowed with the Wasserstein metric defined by the rule: if $m_1,m_2\in \mathcal{P}^p(X)$, then 
	\[W_p(m_1,m_2)\triangleq \Bigg[\inf\Biggl\{\int_{X\times X}\rho_X^p(x_1,x_2)\pi(d(x_1,x_2)):\pi\in \Pi(m_1,m_2)\Biggr\}\Bigg]^{1/p}.\] Hereinafter, $\Pi(m_1,m_2)$ stands for the set of probabilities $\pi$ on $X_1\times X_2$ such that $\operatorname{p}^i\sharp\pi=m_i$, $i=1,2$. The convergence within $W_p$ implies the narrow convergence, the converse holds true only if $X$ is compact \cite[Proposition 7.1.5]{Ambrosio}. Primary, we will consider the case $p=2$.
	\item The set of all continuous curves in $\rd$ on $[s,r]$ is denoted by $\Gamma_{s,r}\triangleq C([s,r],\rd)$. If  $t\in [s,r]$, the denote by $e_t$ the evaluation operator defined by the rule: for $x(\cdot)\in \Gamma_{s,r}$
	\[e_t(x(\cdot))\triangleq x(t).\]
\end{itemize}

\section{First-order mean field type control system}\label{sect:1_order}
In the paper, we consider a mean field type control system formally described by a continuity equation
\begin{equation}\label{mfcs:eq:continuity}
	\partial_t m(t)+\operatorname{div}(f(t,x,m(t),u(t,x))m(t))=0.
\end{equation} Here, $t\in [0,T]$, $x\in\rd$ is a phase variable, $m(t)$ stands for a current distribution of agents, $u(t,x)$ is a control implemented by an agent at the time $t$ and the state $x$. We will assume that the control is chosen from a set $U$. This system describes the behavior of the infinitely many identical agents such that the dynamics of each agent is determined by the ODE:
\begin{equation}\label{mfcs:eq:ODE}
	\frac{d}{dt}x(t)=f(t,x(t),m(t),u(t)).
\end{equation} 
We impose the following conditions on $U$ and $f$.
\begin{enumerate}[label=(C\arabic*)]
	\item $U$ is a metric compact;
	\item $f$ is a continuous function;
	\item there exists a set $\mathcal{K}\subset \rd$ such that, if $\operatorname{supp}(m)\subset \mathcal{K}$, then
	\begin{equation}\label{mp1:equality:f}
		f(t,x,m,u)=0,\ \ x\notin \mathcal{K};
	\end{equation}
\item $f$ is bounded on $[0,T]\times \mathcal{K}\times \pK\times U$ by a constant $R$;
	\item $f$ is Lipschitz continuous w.r.t. the space and measure variables on $[0,T]\times \mathcal{K}\times \pK\times U$; the Lipschitz constant is denoted by $C_f$.
\end{enumerate} 

It is convenient to use the relaxation of the controls. This means that we replace the set of instantaneous controls $U$ with the set of probabilities on $U$. The time-dependent relaxed controls are defined as follows.

For $s,r\in [0,T]$, $s<r$, put \[\mathcal{U}_{s,r}\triangleq \Lambda([s,r],\lambda;U),\] where $\lambda$ is the Lebesgue measure on $[s,r]$. An element of  $\mathcal{U}_{s,r}$ is a  control measure on $[s,r]$.  Notice that $\mathcal{U}_{s,r}$ is a compact subset of $\mathcal{M}([s,r]\times U)$. For now, assume that we are given with a flow of probabilities $m(\cdot):[s,r]\rightarrow\mathcal{P}^p(\rd)$. Furthermore, let $y\in\rd$ be an initial state, and let $\xi\in \mathcal{U}_{s,r}$ be a relaxed control. Then, the corresponding motion of an agent is a solution of the initial value problem
\begin{equation}\label{mfcs:eq:relaxed}
	\frac{d}{dt}x(t)=\int_U f(t,x(t),m(t),u)\xi(du|t),\ \ x(s)=y.
\end{equation} We denote this motion by $x(\cdot,s,y,m(\cdot),\xi)$. Let us denote by $\operatorname{traj}^{s,r}_{m(\cdot)}$  the operator assigning to $y\in \rd$ and $\xi\in\mathcal{U}_{s,r}$ the trajectory  $x(\cdot,s,y,m(\cdot),\xi)\in \Gamma_{s,r}$. 

Further, for $m\in\mathcal{P}^p(\rd)$, put
\[\mathcal{A}_{s,r}[m]\triangleq \Lambda(\rd,m;\mathcal{U}_{s,r}).\] The set $\mathcal{A}_{s,r}[m]$ is the set of distributions of pairs consisting of an initial state and a relaxed control compatible with the initial probability $m$. 

\begin{definition}\label{def:flow} Let $s,r\in [0,T]$, $s<r$, $m_*\in\mathcal{P}^p(\rd)$, $\alpha\in\mathcal{A}[m_*]$. We say that a function $m(\cdot):[s,r]\rightarrow \mathcal{P}^p(\rd)$ is a motion of the deterministic mean field type system~\eqref{mfcs:eq:continuity} produced by the initial time $s$ and the distribution of controls $\alpha\in \mathcal{A}_{s,r}[m_*]$ if there exists a measure $\chi\in \mathcal{P}^2(\Gamma_{s,r})$ such that 
	\begin{itemize}
		\item $m(s)=m_*$;
		\item $\chi=\operatorname{traj}^{s,r}_{m(\cdot)}\sharp \alpha$;
		\item $m(t)=e_t\sharp\chi$ on $[s,r]$.
	\end{itemize}
\end{definition} Below we denote the motion of the system~\eqref{mfcs:eq:continuity} produced by the initial time $s$ and the distribution of controls $\alpha\in\mathcal{A}_{s,r}[m_*]$ by $m(\cdot,s,\alpha)$.

\begin{proposition}\label{prop:existence} For each $s,r\in [0,T]$, $r>s$, $m_*\in\pK$, $\mathcal{A}_{s,r}[m_*]$, there exists a unique motion $m(\cdot,s,\alpha)$. Moreover, $m(t,s,\alpha)\in\pK$ for all $t\in [s,r]$.
\end{proposition}
This proposition in fact is proved in~\cite{Averboukh2022}.

Further, we introduce the concatenation of distribution of controls. First, if $s_0,s_1,s_2\in [0,T]$, $s_0<s_1<s_2$, $\xi_0\in\mathcal{U}_{s_0,s_1}$, $\xi_1\in\mathcal{U}_{s_1,s_2}$, then the concatenation $\xi\triangleq \xi_0\diamond_{s_1}\xi_1$ of these controls is defined by its disintegration w.r.t. the Lebesgue measure:
\[
\xi(\cdot|t)\triangleq \left\{\begin{array}{cc}
	\xi_0(\cdot|t), & t\in [s_0,s_1),\\
	\xi_1(\cdot|t), & t\in [s_1,s_2].
\end{array}\right.
\]
\begin{definition}
Let 
\begin{itemize}
	\item $s_0,s_1,s_2\in [0,T]$, $s_0<s_1<s_2$;
	\item $m_0,m_1\in \prd$;
	\item $\alpha_0\in\mathcal{A}_{s_0,s_1}[m_0]$, $\alpha_1\in\mathcal{A}_{s_1,s_2}[m_1]$.
\end{itemize} be such that 
\[m_1=m(s_1,s_0,\alpha_0).\] A probability $\alpha\in\mathcal{A}_{s_0,s_2}[m_0]$ defined by the rule: for every $\phi\in C_b(\rd\times\mathcal{U}_{s_0,s_2})$,
\[\begin{split}
\int_{\rd\times\mathcal{U}_{s_0,s_2}}\phi(&y,\xi)\alpha(d(y,\xi))\\&\triangleq \int_{\rd\times\mathcal{U}_{s_0,s_1}}\int_{\mathcal{U}_{s_1,s_2}}\phi(y,\xi_0\diamond_{s_1}\xi_1)\alpha_1(d\xi_1|x^0(y,\xi_0))\alpha_0(d(y,\xi_0))\end{split}\] is called a concatination of distribution $\alpha_0$ and $\alpha_1$. Here we use the designations
\[x^0(y,\xi_0)\triangleq x(s_1,s_0,y,m^0(\cdot),\xi_0),\ \  m^0(\cdot)\triangleq m(\cdot,s_0,y,m_0,\alpha_0).\] With some abuse of notation, we denote the concatenation of distributions by $\alpha_0\diamond_{s_1}\alpha_1$.
\end{definition}

\begin{proposition}\label{prop:concatination} Assume that $s_0<s_1<s_2$, $m_0,m_1\in\pK$, $\alpha_0\in\mathcal{A}_{s_0,s_1}[m_0]$, $\alpha_1\in\mathcal{A}_{s_1,s_2}[m_1]$ are such that \[m_1=m(s_1,s_0,\alpha_0).\] Then,
\begin{itemize}
	\item $m(t,s_0,\alpha_0\diamond_{s_1}\alpha_1)=m(t,s_0,\alpha_0)$ when $t\in [s_0,s_1]$;
	\item $m(t,s_0,\alpha_0\diamond_{s_1}\alpha_1)=m(t,s_1,\alpha_1)$ when $t\in [s_1,s_2]$.
\end{itemize}
\end{proposition}
This proposition directly follows from the definition of concatenation.

Let us complete this section with the equivalent formalization  of the motion in mean field type control system \eqref{mfcs:eq:continuity}. 
The following result is proved in \cite[Theorem 1]{Jimenez_Marigonda_Quincampoix}.
\begin{proposition} Let $m(\cdot):[s,r]\rightarrow\pK$ be equal to $m(\cdot,s,\alpha)$ for some distribution of controls $\alpha$. Then there exists a velocity field $v:[s,r]\times\rd\rightarrow \rd$ such that 
	\begin{enumerate}[label=(V\arabic*)]
		\item\label{mftc:cond:V1} $v(t,x)\in\operatorname{co}\{f(t,x,m(t),u):u\in U\}$ for a.e. $t\in [s,r]$, $m(t)$-a.e. $x\in\rd$;
		\item\label{mftc:cond:V2} the equation 
		\[\partial_t m(t)+\operatorname{div}(v(t,x)m(t))=0\] holds in the sense of distributions, i.e.,
		for every $\phi\in C_c^1([s,r]\times\rd)$,
		\[\int_s^r\int_{\rd}[\partial_t\phi(t,x)+\langle \nabla\phi(t,x),v(t,x)\rangle ]m(t,dx)dt=0.\]
	\end{enumerate}
Conversely, if $m(\cdot)$ and $v(\cdot,\cdot)$ satisfy conditions~\ref{mftc:cond:V1},~\ref{mftc:cond:V2}, then there exits a distribution of controls $\alpha$ such that 
\[m(\cdot)=m(\cdot,s,\alpha).\]
\end{proposition}

\section{Mean filed  Markov chains}\label{sect:Markov_chain}
In this section, we introduce a controlled nonlinear Markov chain. Let 
$\mathcal{S}\subset\mathcal{K}$ be at most countable set. Distributions on $\mathcal{S}$ are described by sequences $\mu=(\mu_{\bar{x}})_{\bar{x}\in\mathcal{S}}$ such that 
\[\mu_{\bar{x}}\geq 0,\ \ \sum_{\bar{x}\in\mathcal{S}}\mu_{\bar{x}}=1.\] The set of such distributions is simplex on $\mathcal{S}$ denoted below by $\Sigma$. Furthermore, let $\Sigma^2$ be a set of sequences $\mu=(\mu_{\bar{x}})_{\bar{x}\in\mathcal{S}}$ such that 
\[\sum_{\bar{x}\in\mathcal{S}}\|\bar{x}\|^2\mu_{\bar{x}}<\infty.\] If $\mathcal{S}$ is finite, the sets $\Sigma$ and $\Sigma^2$ coincide. There is a natural isomorphism between $\Sigma$ and $\mathcal{P}(\mathcal{S})$:
\[\mu=(\mu_{\bar{x}})_{\bar{x}\in\mathcal{S}}\mapsto \mathscr{I}(\mu)\triangleq \sum_{\bar{x}\in\mathcal{S}}\delta_{\bar{x}}\mu_{\bar{x}}.\] Hereinafter, $\delta_{z}$ stands for the Dirac measure concentrated at $z$.  

For $t\in [0,T]$, $\mu\in\Sigma^2$, $u\in U$, let $Q(t,\mu,u)=(Q_{\bar{x},\bar{y}}(t,\mu,u))_{\bar{x},\bar{y}\in\mathcal{S}}$ be a Kolmogorov matrix, i.e., $Q_{\bar{x},\bar{y}}(t,\mu,u)\geq 0$ when $\bar{x}\neq\bar{y}$ and, for each $\bar{x}\in\mathcal{S}$,
\[\sum_{\bar{y}\in\mathcal{S}}Q_{\bar{x},\bar{y}}(t,\mu,u)=0.\] Now let us introduce a mean field Markov chain generated by this Kolmogorov matrix. To this end, we first consider a relaxation of the control space. As above, a relaxed control on $[s,r]$ is an element of $\mathcal{U}_{s,r}$. We will use the feedback approach. This means that we are given with a sequence of relaxed control $\zeta_{\mathcal{S}}=(\zeta_{\bar{x}})_{\bar{x}\in\mathcal{S}}$. Notice that the set of feedback controls is $\mathcal{U}^{\mathcal{S}}_{s,r}$.  In this case, the instantaneous probability rate for transition from $\bar{x}$ to $\bar{y}$ is equal to
\[\mathcal{Q}_{\bar{x}\,\bar{y}}(t,\mu(t),\zeta_\mathcal{S})\triangleq\int_U Q_{\bar{x},\bar{y}}(t,\mu,u)\zeta_{\bar{x}}(du|t).\] Moreover, we denote  \[\mathcal{Q}(t,\mu(t),\zeta_\mathcal{S})\triangleq (\mathcal{Q}_{\bar{x}\,\bar{y}}(t,\mu(t),\zeta_\mathcal{S}))_{\bar{x},\bar{y}\in\mathcal{S}}.\]

\begin{definition}\label{def:Markov} Given a time interval $[s,r]$, an initial distribution of states $\mu_*\in\Sigma^2$ and a feedback control $\zeta_{\mathcal{S}}=(\zeta_{\bar{x}})_{\bar{x}\in\mathcal{S}}$, we say that $\mu(\cdot)$ is a motion in the mean field Markov chain if it satisfies the following initial value problem:
	\begin{equation}\label{markov:eq:Kolmogorov_x}
		\frac{d}{dt}\mu_{\bar{y}}(t)=\sum_{\bar{x}\in\mathcal{S}} \mu_{\bar{x}}(t)\mathcal{Q}_{\bar{x},\bar{y}}(t,\mu(t),\zeta_{\mathcal{S}}),\ \ \mu_{\bar{y}}(s)=\mu_{*,\bar{y}}.
	\end{equation}
	\end{definition}

Notice that system~\eqref{markov:eq:Kolmogorov_x} can be rewritten in the vector form
\begin{equation}\label{markov:eq:Kolmogorov_system}
	\frac{d}{dt}\mu(t)=\mu(t)\mathcal{Q}(t,\mu(t),\zeta_\mathcal{S}).
\end{equation} 

To guarantee the existence of the distribution $\mu(\cdot)$, it is sufficient to assume that $Q$ has continuous entries, for each $\bar{x}$ only finite number of entries $Q_{\bar{x},\bar{y}}(t,\mu,u)$ are non-zero and the dependence of the matrix $Q$ on $\mu$ is Lipschitz continuous.

Let us also give a probabilistic interpretation of Definition~\ref{def:Markov}. Set
\begin{itemize}
	\item $\Omega_{s,r}\triangleq D([s,r];\mathcal{S})$, where $D([s,r];\mathcal{S})$ stands a Skorokhod space of c\`adl\`ag functions;
	\item $\mathcal{F}_{s,r}\triangleq \mathcal{B}(D([s,r];\mathcal{S}))$;
	\item $\mathbb{F}_{s,r}= \{\mathcal{F}_{s,r}^t\}_{t\in [s,r]}$, where $\mathcal{F}_{s,r}^t\subset \mathcal{F}_{s,r}$ is a $\sigma$-algebra such that  projections of its elements on $[s,t]$ form the whole $\sigma$-algebra $\mathcal{B}(D([s,t];\mathcal{S}))$, while the projection on $[t,r]$ is a trivial $\sigma$-algebra;
	\item $X(t,\omega)\triangleq \omega(t)$.
\end{itemize}  Further, we define the generator $L_t[\mu,\zeta_{\mathcal{S}}]$ by the rule: for $\phi\in C_b(\mathcal{S})$,
\[L_t[\mu,\zeta_{\mathcal{S}}]\phi(x)\triangleq \sum_{\bar{y}\in\mathcal{S}}\mathcal{Q}_{\bar{x},\bar{y}}(t,\mu,\zeta_{\mathcal{S}})\phi(\bar{y}).\]
\begin{definition}\label{def:representation} Let $[s,r]$ be a time interval, $\mu_*$ be an initial distribution of states, $\zeta_{\mathcal{S}}\in\mathcal{U}_{s,r}^\mathcal{S}$ and let $\mu(\cdot)=\mu(\cdot,s,\mu_*,\zeta_{\mathcal{S}})$.  We say that a probability $\mathbb{P}_{s,r}$ on $\mathcal{F}_{s,r}$ realizes $\mu(\cdot)$ if
	\begin{itemize}
		\item $\mu_{\bar{x}}(t)=\mathbb{P}_{s,r}\{\omega(t)=\bar{x}\}$;
		\item for each $\phi\in C_b(\mathcal{S})$, the process
		\[\phi(X(t))-\int_s^t L_t[\mu(t),\zeta_{\mathcal{S}}]\phi(X(t'))dt'\] is a $\mathbb{F}_{s,r}$-martingale.
	\end{itemize}
\end{definition}
Below, if $\mathbb{P}_{s,r}$ is a realization of $\mu(\cdot)$, $\mathbb{E}_{s,r}$ stands for the corresponding expectation.

Notice that there exists at least one representation of flow $\mu(\cdot)$ \cite[Theorem 5.4.2]{Kolokoltsov}.

The main result of the paper is proved under the following approximation assumptions.
\begin{enumerate}[label=(A\arabic*)]
	\item\label{markov:cond:space} \[\max_{x\in\mathcal{K}}\min_{\bar{y}\in\mathcal{S}}\|x-\bar{y}\|\leq \varepsilon;\]
	\item\label{markov:cond:matrix} entries of the matrix $Q$ are uniformly bounded by a number $B_Q$.
	\item\label{markov:cond:f} for each $t\in [0,T]$, $\bar{x}\in\mathcal{S}$, $\mu\in\Sigma^2$ and $u\in U$,
	\[\Bigl\|f(t,\bar{x},\mathscr{I}(\mu),u)-\sum_{\bar{y}\in\mathcal{S}}(\bar{y}-\bar{x})Q_{\bar{x},\bar{y}}(t,\mu,u)\Bigl\|\leq \varepsilon;\]
	\item\label{markov:cond:variation} for each $t\in [0,T]$, $\bar{x}\in\mathcal{S}$, $\mu\in\Sigma^2$ and $u\in U$,
	\[\sum_{\bar{y}\in\mathcal{S}}\|\bar{y}-\bar{x}\|^2Q_{\bar{x},\bar{y}}(t,\mu,u)\leq \varepsilon^2.\]
\end{enumerate}

Without loss of generality, we assume that $\varepsilon\leq 1$.

An example of the Markov chain that satisfies  assumptions \ref{markov:cond:space}--\ref{markov:cond:variation} can be constructed on a regular lattice as follows. First, we represent $f$ in the coordinate-wise form
\[f(t,x,m,u)=(f_i(t,x,m,u))_{i=1}^d.\]
Let $h>0$, $\mathcal{K}^h\triangleq \mathcal{K}+\mathbb{B}_h$, where $\mathbb{B}_h$ is a ball centered at the origin and radius $h$. We put
\begin{equation}\label{markov:intro:S_h}\mathcal{S}\triangleq \mathcal{K}^h\cap h\mathbb{Z},\end{equation}
\begin{equation}\label{markov:intro:Q_h}Q_{\bar{x},\bar{y}}(t,\mu,u)\triangleq \left\{\begin{array}{ll}
	h^{-1}|f_i(t,\bar{x},\mathscr{I}(\mu),u)|, & \bar{y}=\bar{x}+h\operatorname{sgn}(f_i(t,\bar{x},\mathscr{I}(\mu),u)),\\
	-h^{-1}\sum_{j=1}^{d}|f_j(t,\bar{x},\mathscr{I}(\mu),u)|, & \bar{y}=\bar{x},\\
	0,& \text{otherwise}.
\end{array}\right.\end{equation} One can directly show that this Markov chain satisfies conditions~\ref{markov:cond:space}--\ref{markov:cond:variation} with $B_Q=dRh^{-1}$ and $\varepsilon\triangleq \max\{h,\sqrt{hdR}\}$. 

Notice that if $\mathcal{K}$ is compact, the phase space for Markov chain introduced by rules~\eqref{markov:intro:S_h},~\eqref{markov:intro:Q_h} is finite and equation~\eqref{markov:eq:Kolmogorov_system} is a system of the finite number of ODEs.

\section{Model predictive control of the first-order mean field type system}\label{sect:model_for_1_order}
In this section, we show that a feedback control in the mean filed type Markov chain can be used directly  to construct a motion in the first order mean filed type control system.

Let $\mu_0\in\Sigma^2$, $\zeta_{\mathcal{S}}\in\mathcal{U}^\mathcal{S}_{0,T}$. Notice that there exists a unique motion in the mean field Markov chain on the time interval $[0,T]$ produced by the control $\zeta_{\mathcal{S}}$ and the initial distribution $\mu_0$. For shortness, we denote it by $\mu(\cdot)$. Below, let $\hat{\zeta}$ assign to $x\in\mathcal{K}$ and $\bar{x}\in\mathcal{S}$ a measure $\zeta_{\bar{x}}\in \mathcal{U}_{0,T}$. Further, let  $m_0\in\pK$, $\Delta=\{s_i\}_{i=0}^n$ be a partition of $[0,T]$. A model predictive strategy for the first order mean field game is constructed as follows.
\begin{enumerate}[label=(D\arabic*)]
	\item\label{mp1:cond:initial} Let $\pi_0$ be an optimal plan between $m_0$ and $\mathscr{I}(\mu_0)$. We define 
	\[\alpha_0\triangleq (\operatorname{p}^1,\hat{\zeta})\sharp\pi_0.\]
	\item\label{mp1:cond:step} Assume now that we already construct controls $\alpha_i$, $i=0,\ldots,k-1$, and a flow of probabilities $m(\cdot)$ on $[0,s_k]$ such that 
	\begin{equation*} \alpha_i\in\mathcal{A}_{s_{i},s_{i+1}}[m(s_i)],\ \ i=0,\ldots,k-1,
	\end{equation*}
\[m(t)=m(t,0,\alpha_0\diamond_{s_1}\ldots\diamond_{s_{k-1}}\alpha_{k-1}),\ \ t\in [0,s_k].\]
	Set $m_k\triangleq m(s_{k},0,\alpha_0\diamond_{s_1}\ldots \diamond_{s_{k-1}}\alpha_{k-1})$ and choose $\pi_k$ to be an optimal plan between $m_k$ and $\mathscr{I}(\mu(s_k))$. As above $\pi_k(\cdot|x)$ is a disintegration of this plan w.r.t. $m_k$. We put 
	\[\alpha_k\triangleq(\operatorname{p}^1,\hat{\zeta})\sharp\pi_k\] and, for $t\in [s_k,s_{k+1}]$,
	\[m(t)\triangleq m(s_{k},0,\alpha_0\diamond_{s_1}\ldots \diamond_{s_{k-1}}\alpha_{k-1}\diamond_{s_k}\alpha_{k}).\]
\end{enumerate} 

\begin{theorem}\label{th:approx_mp1} If $\Delta=\{s_i\}_{i=0}^n$ is a partition of $[0,T]$, with $d(\Delta)\leq 1$, while $\alpha_0,\ldots,\alpha_{n-1}$ are constructed by the rules~\ref{mp1:cond:initial},~\ref{mp1:cond:step} and $m(\cdot)\triangleq m(\cdot,0,\alpha_0\diamond_{s_1}\ldots\diamond_{s_{n-1}}\alpha_{n-1})$, then 
	\[\begin{split}
	W_p(m(t),\mathscr{I}(\mu(t)))\leq C_0W_2(\mathscr(\mu_0),m_0)+C_1&\varepsilon+C_2d^{1/2}(\Delta)\\&+C_3d(\Delta)+C_4\varepsilon B_Qd(\Delta)+C_5B_Qd^2(\Delta).\end{split}\] Here $C_0,\ldots, C_5$ are constants those depend only on $f$ and $T$.
\end{theorem}

\begin{lemma}\label{lm:X_squared} Let $\mu(\cdot)$ be a distribution of agents in the mean field Markov chain, $\mathbb{P}_{s,r}$ be its realization. Then,
	\[\mathbb{E}_{s,r}(\|X(t)-X(s)\|^2|X(s)=\bar{z})\leq \varepsilon^2(t-s)+C'_1(t-s)^{3/2},\] where
	\[C'_1\triangleq  4(R+1)e^{2(R+1)T}/3.\]
\end{lemma}
\begin{proof}
For fixed $\bar{z}\in\mathcal{S}$, let us denote \[q_{\bar{z}}(\bar{x})\triangleq \|\bar{x}-\bar{z}\|^2.\] We have that 
\[\begin{split}
L_t[\mu,u]q_{\bar{z}}(\bar{x})&=\sum_{\bar{y}\in\mathcal{S}}Q_{\bar{x},\bar{y}}(t,\mu,u)\|\bar{y}-\bar{z}\|^2\\&= \sum_{\bar{y}\in\mathcal{S}}Q_{\bar{x},\bar{y}}(t,\mu,u)(\|\bar{y}-\bar{x}\|^2+\|\bar{x}-\bar{z}\|^2+2\langle \bar{y}-\bar{x},\bar{x}-\bar{z}\rangle)\\&= \sum_{\bar{y}\in\mathcal{S}}Q_{\bar{x},\bar{y}}(t,\mu,u)\|\bar{y}-\bar{x}\|^2+\Bigl\langle \sum_{\bar{y}\in\mathcal{S}}Q_{\bar{x},\bar{y}}(t,\mu,u)(\bar{y}-\bar{x}),\bar{x}-\bar{z}\Bigr\rangle.
\end{split}\]	
Due to assumption~\ref{markov:cond:variation}, we have that the first term is bounded by $\varepsilon^2$. Moreover, condition~\ref{markov:cond:f} implies that 
\[\Bigg\|\sum_{\bar{y}\in\mathcal{S}}Q_{\bar{x},\bar{y}}(t,\mu,u)(\bar{y}-\bar{x})\Bigg\|\leq R+\varepsilon.\] 	
Using these estimates and Definition~\ref{def:representation}, we conclude that  that 
 \begin{equation}\label{mp1:ineq:E_X_t_s}
 	\begin{split}
 	\mathbb{E}_{s,rT}(\|X(t)&-X(s)\|^2|X(s)=\bar{z})\\&\leq \varepsilon^2(t-s)+2\int_{s}^t(R+\varepsilon)(\mathbb{E}_{0,T}\|X(t')-X(s)\||X(s)=\bar{z})dt'.
\end{split} 
\end{equation} Since we assume that $\varepsilon\leq 1$, we have that 
\[\mathbb{E}_{0,T}(\|X(t)-X(s)\|^2|X(s)=\bar{z})\leq C''_1(t-s),\] where
\[C''_1\triangleq e^{2(R+1)T}.\] Plugging this estimate to~\eqref{mp1:ineq:E_X_t_s}, we obtain the statement of the lemma.
\end{proof}

\begin{lemma}\label{lm:nu_change}
	Assume that $s,r\in [0,T]$, $s<r$, $\nu_*=(\nu_{*,\bar{x}})_{\bar{x}\in\mathcal{S}}\in\Sigma^2$, $\zeta_{\mathcal{S}}\in\mathcal{U}^{\mathcal{S}}_{0,T}$, $\mu(\cdot):[0,T]\rightarrow\Sigma^2$, while $\nu(\cdot)$ satisfies
	\begin{equation}\label{mp1:eq:nu}
		\frac{d}{dt}\nu(t)=\nu(t)Q(t,\mu(t),\zeta_{\mathcal{S}}(t)), \nu(s)=\nu_*.
	\end{equation} Then, for each $\bar{x}\in\mathcal{S}$,
\[|\nu_{\bar{x}}(t)-\nu_{*,\bar{x}}|\leq B_Q(t-s).\]
\end{lemma}
\begin{proof}
	We have that 
	\[|\nu_{\bar{x}}(t)-\nu_{*,\bar{x}}|\leq\int_s^t\int_U\sum_{\bar{y}\in\mathcal{S}}\nu_{\bar{y}}(t')\big|Q_{\bar{y},\bar{x}}(t',\mu(t),u)\big|\zeta_{\bar{y}}(du|t'))dt'\leq B_Q(t-s).\]
\end{proof}

\begin{proof}[Proof of Theorem~\ref{th:approx_mp1}]
Let us estimate the squared Wasserstein distance $W_2^2(\mu(t),m(t))$ for $t\in [s_k,s_{k+1}]$, $k=0,\ldots,n-1$. Let $x_{k}(\cdot,y,\bar{z})$ be a solution on $[s_k,s_{k+1}]$ of the differential equation
\[\frac{d}{dt}x(t)=\int_{U} f(t,x(t),m(t),u)\zeta_{\bar{z}}(du|t),\ \ x(s_k)=y.
\] Notice that $m(t)=x_{k}(t,\cdot,\cdot)\sharp\pi_{k}$, while by construction $\mathscr{I}(\mu(t))\triangleq X(t)\sharp \mathbb{P}_{s_k,s_{k+1}}$. Recall that $\mathbb{P}_{s_k,s_{k+1}}$ is a probability that realizes a flow of distributions $\mu(\cdot)$ on $[s_k,s_{k+1}]$, whereas $\mathbb{E}_{s_k,s_{k+1}}$ is the corresponding expectation. Thus, we have that 
\begin{equation}\label{mp1:ineq:W_2_2}
	W_2^2(\mathscr{I}(\mu(t)),m(t))=\int_{\mathcal{K}\times\mathcal{S}}\mathbb{E}_{s_k,s_{k+1}} \big(\|X(t)-x_k(t,y,\bar{z})\|^2|X(s_k)=\bar{z}\big)\pi_{k}(d(y,\bar{z})).
\end{equation}
Now, let us evaluate the quantity $\mathbb{E}_{s_k,s_{k+1}} (\|X(t)-x_k(t,y,\bar{z})\|^2|X(s_k)=\bar{z})$. First, notice that 
\begin{equation*}
	\begin{split}
		\|X(t)-x_k(t&,y,\bar{z})\|^2\leq \|X(s_k)-x_k(s_k,y,\bar{z})\|^2\\ +2\|X(&t)-X(s_k)\|^2+2\|x_k(t,y,\bar{z})-x_k(s_k,y,\bar{z})\|^2\\&+
		2\langle X(t)-X(s_k),X(s_k)-x_k(s_k,y,\bar{z})\rangle\\&-2\langle x_k(t,y,\bar{z})-x_k(s_k,y,\bar{z}),X(s_k)-x_k(s_k,y,\bar{z})\rangle.
	\end{split}
\end{equation*} Thus, 
\begin{equation}\label{mp1:equality:X_t_s}
	\begin{split}
		\mathbb{E}_{s_k,s_{k+1}}\big(\|X(t)-x_k&(t,y,\bar{z})\|^2|X(s_k)=\bar{z}\big)\leq \|y-\bar{z}\|^2\\+ 2\mathbb{E}_{s_k,s_{k+1}}&\big(\|X(t)-\bar{z}\|^2|X(s_k)=\bar{z}\big)+2\|x_k(t,y,\bar{z})-y\|^2\\+
		\big\langle&\mathbb{E}_{s_k,s_{k+1}}( X(t)-\bar{z}|X(s_k)=\bar{z}),\bar{z}-y\big\rangle-\langle x_k(t,y,\bar{z})-y,\bar{z}-y\rangle.
	\end{split}
\end{equation} Due to Lemma~\ref{lm:X_squared}, and the boundness of $f$, we have that
\begin{equation}\label{mp1:ineq:squared_terms}
	\begin{split}
		2\mathbb{E}_{s_k,s_{k+1}}\big(\|X(t)-\bar{z}&\|^2|X(s_k)=\bar{z}\big)+2\|x_k(t,y,\bar{z})-y\|^2\\ &\leq 2\varepsilon^2(t-s)+2C'_1(t-s)^{3/2}+R^2(t-s)^2.
	\end{split}
\end{equation} Further,
\begin{equation*}
\mathbb{E}_{s_k,s_{k+1}}( X(t)-\bar{z}|X(s_k)=\bar{z})=\mathbb{E}_{s_k,s_{k+1}}\Bigg(\int_{s_k}^t \sum_{\bar{y}\in\mathcal{S}} Q_{X(t'),y}(y-X(t'))|X(s)=\bar{z}\Bigg)dt'.
\end{equation*} This and condition~\ref{markov:cond:f} yield that 
\begin{equation*}
\begin{split}
	\Bigg\|\mathbb{E}_{s_k,s_{k+1}}( X(t)-\bar{z}|X(&s_k)=\bar{z})\\-\int_{s_k}^t \mathbb{E}_{s_k,s_{k+1}}\Bigg(&\int_Uf(t',X(t'),\mathscr{I}(\mu(t)),u)\zeta_{X(t')}(du|t')|X(s)=\bar{z}\Bigg)dt'\Bigg\|\\&{}\hspace{210pt}\leq \varepsilon(t-s_k).
\end{split} 
\end{equation*} Recall that
\[
\begin{split}
	\mathbb{E}_{s_k,s_{k+1}}\Bigg(\int_Uf(t',X(t'),&\mathscr{I}(\mu(t)),u)\zeta_{X(t')}(du|t')|X(s)=\bar{z}\Bigg)dt'\\&=
	\sum_{\bar{x}\in\mathcal{S}}\int_Uf(t',\bar{x},\mathscr{I}(\mu(t)),u)\zeta_{X(t')}(du|t')\nu_{*,\bar{x}}.
\end{split}
\]
 Thus, due to Lemma~\ref{lm:nu_change}, we have that 
\begin{equation*}
	\begin{split}
		\Bigg\|\mathbb{E}_{s_k,s_{k+1}}( X(t)-\bar{z}|X(s_k)=\bar{z})-\int_{s_k}^t\int_Uf(t',\bar{z},&\mathscr{I}(\mu(t)),u)\zeta_{\bar{z}}(du|t')dt'\Bigg\|\\&\leq \varepsilon(t-s_k)+RB_Q(t-s_k)^2.
	\end{split} 
\end{equation*}
This and the definition of the motion $x(\cdot,y,\bar{z})$ imply that 
\begin{equation*}
	\begin{split}
		\Bigg|\big\langle\mathbb{E}_{s_k,s_{k+1}}( &X(t)-\bar{z}|X(s_k)=\bar{z}),\bar{z}-y\big\rangle-\langle x_k(t,y,\bar{z})-y,\bar{z}-y\rangle\Bigg| \\ &\leq 
		\int_{s_k}^t\int_U\|f(t',\bar{z},\mathscr{I}(\mu(t)),u)-f(t',x(t',y,\bar{z}),m(t),u)\|\zeta_{\bar{z}}(du|t')dt'\cdot
		\|y-\bar{z}\|\\ &{}\hspace{200pt}+(\varepsilon(t-s_k)+RB_Q(t-s_k)^2)\|y-\bar{z}\|.
	\end{split}
\end{equation*} 
Using the Lipschitz continuity of the function $f$, we obtain the following:
\begin{equation*}
	\begin{split}
		\Bigg|\big\langle\mathbb{E}_{s_k,s_{k+1}}( X(t)-\bar{z}|X(s_k)=\bar{z}),&\bar{z}-y\big\rangle-\langle x_k(t,y,\bar{z})-y,\bar{z}-y\rangle\Bigg| \\ \leq 
		&C_fR(t-s_k)^2\|y-\bar{z}\|\\&+C_f\int_{s_k}^tW_2(\mathscr{I}(\mu(t')),m(t'))\|y-\bar{z}\|\\&+(\varepsilon(t-s_k)+RB_Q(t-s_k)^2)\|y-\bar{z}\|.
	\end{split}
\end{equation*} Plugging this estimate into~\eqref{mp1:equality:X_t_s} and taking into account~\eqref{mp1:ineq:squared_terms}, we derive the estimate:
\[
\begin{split}
	\mathbb{E}_{s_k,s_{k+1}}\big(\|X(t)-x_k(t&,y,\bar{z})\|^2|X(s_k)=\bar{z}\big)\leq \|y-\bar{z}\|^2\\&+2\varepsilon^2(t-s_k) +2C'_1(t-s)^{3/2}+R(t-s_k)^2\\&+C_fR(t-s_k)^2+C_fR(t-s_k)\|y-\bar{z}\|^2\\& +C_f(t-s_k)W_2^2(\mathscr{I}(\mu(s_k)),m(s_k)) +3C_f(t-s_k)\|y-\bar{z}\|^2\\&+C_f\varepsilon^2(t-s_k)^2+C'_1C_f(t-s_k)^{5/2}+RC_f(t-s_k)^2\\&+
	(\varepsilon+RB_Q(t-s_k))^2(t-s_k)+(t-s_k)\|y-\bar{z}\|^2\\ \leq
	\|y-\bar{z}\|^2&+C'_2(t-s_k)\|y-\bar{z}\|^2 +C'_3(t-t_k)W_2^2(\mathscr{I}(\mu(s_k)),m(s_k))\\&+3\varepsilon(t-s_k)+2C'_1(t-s_k)^{3/2} +C'_4(t-s_k)^2\\&+C'_5\varepsilon B_Q(t-s_k)^2+R^2B_Q^2(t-s_k)^3.
\end{split}
\] Due to~\eqref{mp1:ineq:W_2_2}, we arrive at the inequality
\begin{equation}\label{mp1:ineq:wasserstein_final}
\begin{split}
	W_2^2(\mathscr{I}(\mu(t)),m(t))\leq (1+C'_6&(t-s_k))W_2^2(\mathscr{I}(\mu(s_k)),m(s_k))\\&+3\varepsilon(t-s_k)+2C'_1(t-s_k)^{3/2} +C'_4(t-s_k)^2\\&+C'_5\varepsilon B_Q(t-s_k)^2+R^2B_Q^2(t-s_k)^3.
\end{split}\end{equation} Applying this inequality sequentially, we deduce the statement of the theorem.
\end{proof}

\section{Model predictive control for Markov chains}\label{sect:model_for_Markov_chain}
In this section, given an initial distribution for the deterministic mean field type control system $m_0\in\pK$, a distribution of controls $\alpha\in\mathcal{U}_{0,T}$ such that $\alpha\in\mathcal{A}_{0,T}[m_0]$,  and an initial system for mean field Markov chain $\mu_0\in\Sigma^2$, we construct a feedback strategy $\zeta_{\mathcal{S}}$ such that the corresponding motion of the Markov chain starting at $\mu_0$ approximates the motion $m(\cdot,0,\alpha)$. Within this section, we denote 
\[m(\cdot)\triangleq m(\cdot,0,\alpha).\] Further, for $(y,\xi)\in \rd\times\mathcal{U}_{0,T}$ and $s\in [0,T]$, we put
\[\mathscr{T}^s(y,\xi)\triangleq x(s,0,y,m(\cdot),\xi).\] Notice that, if $(y',\xi')=\mathscr{T}^s(y,\xi)$, then
\[x(\cdot,0,y,m(\cdot),\xi)=x(\cdot,s,y',m(\cdot),\xi').\] Informally, the operator $\mathscr{T}^s$ transfers the initial condition and the control from $t=0$ to the time $s$. 

Below, if $\zeta_{\mathcal{S}},\zeta'_{\mathcal{S}}$ are feedback controls for the Markov chain on $[s,r]$ and $[r,\theta]$ respectively, then we denote by $\zeta_{\mathcal{S}}\diamond_r\zeta_{\mathcal{S}}'$ the feedback control such that $\zeta_{\mathcal{S}}\diamond_r\zeta_{\mathcal{S}}'=(\zeta_{\bar{x}}\diamond_r\zeta_{\bar{x}}')_{\bar{x}\in\mathcal{S}}$.

Finally, let $\Delta=\{s_i\}_{i=0}^n$ be a partition of $[0,T]$.

 The construction is stepwise.
\begin{enumerate}[label=(M\arabic*)]
	\item Let $\pi_0$ be an optimal plan between $m_0$ and $\mathscr{I}(\mu_0)$ and let $\pi_0(\cdot|\bar{x})$ be its disintegration w.r.t. $\mathscr{I}(\mu_0)$. We define a probability $\zeta_{\mathcal{S},0}=(\zeta_{\bar{x},0})_{\bar{x}}\in \mathcal{U}_{s_0,s_1}^{\mathcal{S}}$ by the rule: for $\phi\in C_b([s_0,s_1]\times U)$
	\[\begin{split}
	\int_{[s_0,s_1]\times U}\phi(t,u)&\zeta_{\bar{x},0}(d(t,u))\\&\triangleq \int_{\mathcal{K}}\int_{\mathcal{U}_{s_0,s_{1}}}\int_{[s_0,s_1]\times U}\phi(t,u)\xi(d(t,u))\alpha_0(\xi|y)\pi_0(dy|\bar{x}). \end{split}\] 
	Hereinafter, $\alpha_0$ is a restriction of $\alpha$ on $[s_0,s_1]$. 
	\item Assume now that we already constructed controls $\zeta_{\mathcal{S},0},\ldots,\zeta_{\mathcal{S},k-1}$ and put, for $t\in [0,s_k]$,
	\[\mu(t)\triangleq \mu(t,0,\mu_0,\zeta_{\mathcal{S},0}\diamond_{s_1}\ldots\diamond_{s_{k-1}}\zeta_{\mathcal{S},k-1}).\] To extend the control to the next time step, we first introduce $\alpha_{k}$ to be a restriction of the distribution of controls $\mathscr{T}^{s_k}\sharp \alpha$ to the time interval $[s_k,s_{k+1}]$. Further, let $\pi_k$ be an optimal plan between $\mathscr{I}(\mu(s_k))$ and $m(s_k)$. A feedback control $\zeta_{\mathcal{S},k}=(\zeta_{\bar{x},k})_{\bar{x}\in\mathcal{S}}\in\mathcal{U}^{\mathcal{S}}_{[s_k,s_{k+1}]}$ is defined by the rule: if $\phi\in C([s_{k},s_{k+1}]\times U)$, then 
	\[\begin{split}
	\int_{[s_0,s_1]\times U}\phi(t,u)&\zeta_{\bar{x},k}(d(t,u))\\&\triangleq \int_{\mathcal{K}}\int_{\mathcal{U}_{s_k,s_{k+1}}}\int_{[s_k,s_{k+1}]\times U}\phi(t,u)\xi(d(t,u))\alpha_k(d\xi|y)\pi_k(dy|\bar{x}).\end{split} \] 
	 
\end{enumerate}

\begin{theorem}\label{th:MP_markov} Let conditions~\ref{markov:cond:space}--\ref{markov:cond:variation} hold true and let $\mu(\cdot)\triangleq \mu(\cdot,0,\mu_0,\zeta_{\mathcal{S},0}\diamond_{s_1}\ldots\diamond_{s_{n-1}}\zeta_{\mathcal{S},n-1})$, then \[\begin{split}
		W_p(m(t),\mathscr{I}(\mu(t)))\leq C_0W_2(\mathscr{I}(\mu_0)&,m_0)+C_1\varepsilon+C_2d^{1/2}(\Delta)\\&+C_3d(\Delta)+C_4\varepsilon B_Qd(\Delta)+C_5B_Qd^2(\Delta).\end{split}\] Here $C_0,\ldots, C_5$ are the same constant as in Theorem~\ref{th:approx_mp1}.
\end{theorem}
\begin{proof}
	The proof mimics the proof of Theorem~\ref{th:approx_mp1}. We consider a time interval $[s_k,s_{k+1}]$ and choose $\mathbb{P}_{s_k,s_{k+1}}$ to be a realization of the flow $\mu(\cdot)$ on $[s_k,s_{k+1}]$.  If $t\in [s_{k-1},s_k]$, then 
	\begin{equation}\label{mmp:ineq:wasserstein}\begin{split}
	W_2^2(m&(t),\mathscr{I}(\mu(t)))\\&\leq \int_{\mathcal{K}\times\mathcal{S}}\int_{\mathcal{U}_{s_k,s_{k+1}}} \mathbb{E}_{s_k,s_{k+1}}(\|X(t)-x_k(t,y,\xi)\|^2|X(s_k)=\bar{z})\pi_k(d(y,\bar{z})).
	\end{split}
	\end{equation} Here, as above, we denote 
\[x_k(t,y,\xi)\triangleq x(t,s_{k},y,m(\cdot),\xi).\]
Further, we have that 
\begin{equation*}
	\begin{split}
		\mathbb{E}_{s_k,s_{k+1}}&(\|X(t)-x_k(t,y,\xi)\|^2|X(s_k)=\bar{z}) \leq \|\bar{z}-y\|^2\\&+
		2\mathbb{E}_{s_k,s_{k+1}}(\|X(t)-X(s_k)\|^2|X(s_k)=\bar{z})+2\|x_k(t,y,\xi)-y\|^2\\&+
		2\mathbb{E}_{s_k,s_{k+1}}(\langle X(t)-X(s_k),\bar{z}-y\rangle|X(s_k)=\bar{z})\\&-2\langle x_k(t,y,\xi)-y,\bar{z}-y\rangle.
	\end{split}
\end{equation*} Using Lemma~\ref{lm:X_squared}, we conclude that 
\begin{equation}\label{mmp:ineq:squared_terms}
	\begin{split}
		2\mathbb{E}_{s_k,s_{k+1}}\big(\|X(t)-X(s_k)&\|^2|X(t_k)=\bar{z}\big)+2\|x_k(t,y,\bar{z})-y\|^2\\ &\leq 2\varepsilon^2(t-s)+2C'_1(t-s)^{3/2}+R^2(t-s)^2.
	\end{split}
\end{equation}
Further, as in the proof of Theorem~\ref{th:approx_mp1}, we have that 
\begin{equation*}
	\begin{split}
		\Bigg\|\mathbb{E}_{s_k,s_{k+1}}( X(t)-\bar{z}|X(s_k)=\bar{z})-\int_{s_k}^t\int_Uf(t',\bar{z},&\mathscr{I}(\mu(t')),u)\zeta_{\bar{z},k}(du|t')dt'\Bigg\|\\&\leq \varepsilon(t-s_k)+RB_Q(t-s_k)^2.
	\end{split} 
\end{equation*}
Simultaneously, 
\begin{equation*}
	x_k(t,y,\xi)-y=\int_{s_k}^t \int_Uf(t',x_k(t',y,\xi),m(t'),u)\xi(du|t')dt'
\end{equation*} Plugging this into~\ref{mmp:ineq:wasserstein} and taking into account the definition of $\zeta_{\mathcal{S},k}$, we conclude that 
\begin{equation*}
	\begin{split}
		W_2^2(m&(t),\mathscr{I}(\mu(t)))\leq W_2^2(m(s_k),\mathscr{I}(\mu(s_k)))\\&+2\varepsilon^2(t-s)+2C'_1(t-s)^{3/2}+R(t-s)^2+
		\varepsilon(t-s_k)+B_Q(t-s_k)^2\\&+
		\int_{\mathcal{K}\times\mathcal{S}}\int_{\mathcal{U}_{s_k,s_{k+1}}}\Bigg\|\int_U f(t',\bar{z},\mathscr{I}(\mu(t')),u)\xi(du|t')\\ &{}\hspace{100pt}-\int_Uf(t',x_k(t',y,\xi),m(t'),u)\xi(du|t')\Bigg| \cdot \|\bar{z}-y\|.
	\end{split}
\end{equation*}
Estimating the last term as in the proof of Theorem~\ref{th:approx_mp1}, we obtain the inequality:
\begin{equation}\label{mmp:ineq:final}
\begin{split}
	W_2^2(\mathscr{I}(\mu(t)),m(t))\leq (1+C'_6&(t-s_k))W_2^2(\mathscr{I}(\mu(s_k)),m(s_k))\\&+3\varepsilon(t-s_k)+2C'_1(t-s_k)^{3/2} +C'_4(t-s_k)^2\\&+C'_5\varepsilon B_Q(t-s_k)^2+R^2B_Q^2(t-s_k)^3.
\end{split}\end{equation} Here the constants are the same as in~\eqref{mp1:ineq:wasserstein_final}. Applying~\eqref{mmp:ineq:final} sequentially, we arrive at the statement of the theorem.
\end{proof}

\section{Hausdorff distance between bundles of flows}\label{sect:Hausdorff}
In this short section, we consider the bundles of flows of probabilities generated by the original first order mean field type control system and the mean field Markov chain. To define them, let
\[\mathcal{X}(m_0)\triangleq \big\{m(\cdot,0,\alpha):\alpha\in\mathcal{A}_{0,T}[m_0]\big\},\]
\[\mathcal{X}_Q(\mu_0)\triangleq \big\{\mu(\cdot,0,\mu,\zeta_{\mathcal{S}}):\zeta_{\mathcal{S}}\in\mathcal{U}_{0,T}^{\mathcal{S}}\big\}.\]
Notice that $\mathcal{X}(m_0)\subset C([0,T];\prd)$, while $\mathcal{X}_Q\subset C([0,T];\Sigma^2)$

If $m(\cdot)\in C([0,T];\prd)$, $\mu\in C([0,T];\Sigma^2)$, we denote 
\[\mathbbm{d}(m(\cdot),\mu(\cdot))\triangleq \sup_{t\in [0,T]}W_2(m(t),\mathscr{I}(\mu(t))).\] If $\Upsilon_1\subset C([0,T];\prd)$, $\Upsilon_2\subset C([0,T],\Sigma^2)$ are closed sets, then we introduce the Hausdorff distance in the standard way:
\[\mathbbm{H}(\Upsilon_1,\Upsilon_2)\triangleq \max\Bigg\{\sup_{m(\cdot)\in \Upsilon_1}\inf_{\mu(\cdot)\in \Upsilon_2}\mathbbm{d}(m(\cdot),\mu(\cdot)),\sup_{\mu(\cdot)\in \Upsilon_2}\inf_{m(\cdot)\in \Upsilon_1}\mathbbm{d}(m(\cdot),\mu(\cdot))\Bigg\}.\]

The main result of this short section is the following.
\begin{proposition}\label{prop:Hausdorff} Assume that the approximation conditions~\ref{markov:cond:space}--\ref{markov:cond:variation} are in force. Then,
	\[\mathbbm{H}(\mathcal{X}(m_0),\mathcal{X}_Q(\mu_0))\leq  C_0W_2(\mathscr{I}(\mu_0),m_0)+C_1\varepsilon,\] where $C_0$ and $C_1$ are constants dependent only on $f$ and $T$.
\end{proposition}
\begin{proof}
From Theorem~\ref{th:MP_markov}, it follows that, given $m(\cdot)\in \mathcal{X}(m_0)$ and arbitrary $\delta>0$, one can find a partition of $[0,T]$ $\Delta=\{s_i\}_{i=0}^n$ and controls $\zeta_{\mathcal{S}},\ldots,\zeta_{\mathcal{S},n-1}$ such that
\[\mathbbm{d}(m(\cdot),\mu(\cdot))\leq C_0W_2(\mathscr(\mu_0),m_0)+C_1\varepsilon+\delta,\] where $\mu(\cdot)=\mu(\cdot,0,\mu_0,\zeta_{\mathcal{S},0}\diamond_{s_1}\ldots\diamond_{s_{n-1}}\zeta_{\mathcal{\S},n-1})$. Passing to the limit while $\delta\rightarrow 0$, we conclude that 
\[\sup_{m(\cdot)\in \mathcal{X}(m_0)}\inf_{\mu(\cdot)\in \mathcal{X}_Q(\mu_0)}\mathbbm{d}(m(\cdot),\mu(\cdot))\leq C_0W_2(\mathscr{I}(\mu_0),m_0)+C_1\varepsilon.\] The inequality \[\sup_{\mu(\cdot)\in \mathcal{X}_Q(\mu_0)}\inf_{m(\cdot)\in \mathcal{X}(m_0)}\mathbbm{d}(m(\cdot),\mu(\cdot))\leq C_0W_2(\mathscr{I}(\mu_0),m_0)+C_1\varepsilon\] is proved in the similar way using Theorem~\ref{th:approx_mp1}.
\end{proof}

\bibliography{mpc}

\end{document}